\documentclass[reqno,12pt]{amsart}

\usepackage{graphicx,color}
\usepackage{latexsym}
\usepackage{graphicx} 
\usepackage{amsthm,amsmath, amssymb,amsfonts}

\theoremstyle{plain}

\theoremstyle{plain}
\newtheorem{thm}{Theorem}[section]
\newtheorem{cor}[thm]{Corollary}
\newtheorem{lem}[thm]{Lemma}
\newtheorem{prop}[thm]{Proposition}

\newtheorem*{thmA}{Classification Theorem}
\newtheorem*{thm*}{Theorem}

\theoremstyle{definition}
 
\newtheorem{rmk}[thm]{Remark}
\newtheorem{exam}[thm]{Example}

\begin{document}
\raggedbottom
\allowdisplaybreaks

\title[Spheres with minimal equators]{Spheres with minimal equators}

\author{Lucas Ambrozio}

\address{IMPA - Associação Instituto Nacional de Matem\'atica Pura e Aplicada, Rio de Janeiro, RJ, Brazil, 22460-320.}
\email{\texttt{l.ambrozio@impa.br}}

\thanks{L.A. was supported by CNPq (309908/2021-3 and 302815/2025-2 - Bolsa PQ and 406666/2023-7 - Universal)
and by FAPERJ (grant SEI-260003/000534/2023 - Bolsa E-26/200.175/2023 and grant
SEI-260003/001527/2023 - APQ1 E-26/210.319/2023).}

\begin{abstract}
	We survey the classification of the Riemannian metrics on spheres with respect to which all equators are minimal hypersurfaces, and discuss problems related to these geometries.
	
\end{abstract}
\maketitle

{\centering\footnotesize \textit{Dedicated to Paolo Piccione at his 60th birthday}.\par}

\section{Introduction}

	Let $\mathbb{S}^n$ be the $n$-dimensional sphere in the Euclidean space $\mathbb{R}^{n+1}$ that has centre located at the origin and radius one. An \textit{equator} of $\mathbb{S}^n$ is any $(n-1)$-dimensional sphere determined by the intersection of $\mathbb{S}^n$ with some hyperplane of $\mathbb{R}^{n+1}$ that passes through the origin.
	
	When $\mathbb{S}^n$ is endowed with its canonical Riemannian metric $\overline{g}$ of constant sectional curvature equal to one, all equators become totally geodesic hypersurfaces. In particular, in dimensions $n \geq 3$, all equators are \textit{minimal} hypersurfaces of $(\mathbb{S}^n,\overline{g})$. As such, they are critical points of the $(n-1)$-dimensional volume measure defined by $\overline{g}$. 
	
	We are interested in the following problem:
	\vspace{2mm}	
\begin{center}
\textit{To classify and to understand the Riemannian metrics on the sphere $\mathbb{S}^n$ with respect to which all equators are minimal.}
\end{center}
	\vspace{2mm}

	Any solution to this problem must take into account the action of the group of diffeomorphisms of $\mathbb{S}^n$ that maps equators into equators, because this is the group that acts naturally on the set of such Riemannian metrics, by pullback. 
	
	In dimension $n=2$, the problem has a slightly different quality, because it is problem about geodesics. One can find its complete solution already in the work of Eugenio Beltrami \cite{Bel1}: up to scaling and the aforementioned action, only the canonical metric $\overline{g}$ on $\mathbb{S}^2$ has all the great circles as its geodesics. We describe some aspects of the history of Beltrami's discovery, and discuss some developments it inspired, in Appendix A.
	
\subsection{The classification}
	The first part of the problem has a quite satisfactory answer, in the sense that, as well as being rather definitive, it displays these geometries under a new light and motivates further questions.

\begin{thmA} \label{thmcorrespondenceA}
	There exists a $PGL(n+1,\mathbb{R})$-equivariant, one-to-one correspondence between the set of Riemannian metrics on the sphere $\mathbb{S}^n$ with respect to which all equators are minimal and the set of algebraic curvature tensors in $\mathbb{R}^{n+1}$ with positive sectional curvature.
\end{thmA}

	An \textit{algebraic curvature tensor} in $\mathbb{R}^{n+1}$ is a $4$-linear map $R$ with the same algebraic symmetries of the curvature tensor of a Riemannian metric on a $(n+1)$-dimensional manifold. The sectional curvature of $R$ is positive when $R(v,w,v,w)>0$ for every orthonormal pair $v,w \in \mathbb{R}^{n+1}$. 
	
	The Classification Theorem says, in particular, that on each dimension $n$, the Riemannian metrics on the sphere $\mathbb{S}^n$ with respect to which all equators are minimal can be parametrized by points belonging to an open positive cone of a finite dimensional real vector space of dimension $n(n+2)(n+1)^2/12$, and that this parametrization respects an action of the $n(n+2)$-dimensional projective general linear group on these spaces.
	
	We present, motivate and justify the existence of this correspondence (see Section \ref{seccorrespondence}), and discuss the actions of the projective general linear group (see Section \ref{secactions}). Readers who are in a hurry will find the correspondence described explicitly as an algorithm in Section \ref{subsectalgorithm}, Theorem \ref{thmcorrespondenceasalgorithm}.
	
	The core of the proof of the Classification Theorem that we present here is essentially the same one explained in the paper \cite{AmbMarNev} of Fernando Marques, André Neves and the author, see Section 9.4. It is based on a technique developed by Theodor Hangan \cite{Han1}, who classified all metrics on $\mathbb{R}^n$ with respect to which all hyperplanes are minimal in terms of \textit{Killing symmetric two-tensors}. In fact, the Classification Theorem can ultimately be deduced from his theorem, see Remark \ref{rmkkne}.
		
\subsection{Geometric properties} While the Classification Theorem provides an useful parametrization of the space of metrics with minimal equators, it does not inform how to read off, from the parameters, what are the geometric properties of the Riemannian metrics that correspond to them, or how to decide whether any two of these metrics are isometric to each other. Thus, the task of understanding the geometry of these metrics, and of their minimal equators, is not accomplished yet. 

	We discuss the case $n=3$ in some detail (see Section 4). In such low dimension, there are powerful tools at our disposal. For instance, the equators are the only (immersed) minimal two-spheres in any $(\mathbb{S}^3,g)$ with minimal equators, as a consequence of the Uniqueness Theorem of José Gálvez and Pablo Mira \cite{GalMir}. They all have Morse index one, nullity three, and the same area. Moreover, there is a criterion to decide when two of these metrics are isometric, and the computation of their isometry group is a problem in representation theory. 

\subsection{Perspectives} The classification of spheres with minimal equators by algebraic curvature tensors with positive sectional curvature suggests intriguing new problems. The explicit nature of the correspondence means that experiments can be carried out, for instance, to test conjectures about the geometry of index one minimal two-dimensional spheres, or about the generalised Zoll metrics in minimal submanifold theory introduced by Marques, Neves and the author in \cite{AmbMarNev}. It also suggests a novel way to understand and to distinguish algebraic curvature tensors with positive sectional curvature, in terms of the geometric invariants associated to the corresponding metrics with minimal equators. We hope that a deeper understand of Riemannian manifolds with positive sectional curvature could be achieved from that point-of-view.

\subsection{Overview of the paper} In Section \ref{seccorrespondence}, we construct the correspondence of Classification Theorem explicitly, and establish it in a slightly more general form (Theorem \ref{thmcorrespondenceAfull}). In Section \ref{secactions}, we move on to understand the natural equivariance of the correspondence under the natural actions of the general linear group. Section \ref{sec3d} is dedicated to results about the geometry of three-dimensional spheres with minimal equators. In Section \ref{sex.perspectives}, we compile a list of questions that seems to be promising starting points for future research. 

	The paper also contains two appendices. In Appendix A, we trace back the origin of the problem to the work of Beltrami about the representation of geodesics in charts, explain why the case of codimension-one equators is the only one that still leaves open problems, and propose questions about metrics on $\mathbb{R}^n$ with minimal hyperplanes. In Appendix B, we pose ``inverse problems" of a similar nature on projective spaces, and within the context of Reese Harvey and Blaine Lawson's calibrated geometries \cite{HarLaw}.  \\

	\noindent \textbf{Acknowledgements:} It is a great pleasure and honour to write a contribution to the special issue dedicated to Paolo's 60th birthday. Since my times as a student, I could see that Paolo is one of the pillars of our Brazilian Differential Geometry community, and I admired his mathematics, his energy and his charisma. Congratulations, Paolo! 
	
	I am also grateful to Renato Bettiol and all the editors of this Festschrift volume for the invitation to take part of it. My contribution is in the form of this survey, which is a more or less organised account of some ideas and problems that I have been playing with since I started to work with Fernando Codá and André Neves. Most of its main results have appeared elsewhere already, perhaps in a different/generalised form, in our joint works \cite{AmbMarNev} and \cite{AmbMarNev2}. My hope is that the perspective I chose to present the subject here will stimulate further investigations about this intriguing type geometry.
	
	In the last years, I gave talks about metrics with minimal equators (including in Paolo's birthday conference at USP in August 2024!), and I am grateful to the appreciative comments and interesting questions that were made about them by Dan Ketover, Reto Buzano, Pablo Mira, Renato Bettiol, Claudio Gorodsky, Robert Bryant, Luis Florit and Wolfgang Ziller, among others. The curiosity the theme aroused convinced me it was worth making an effort to bring attention to problems related to it, instead of just leaving them implicit in the literature.
	
	Finally, I would like to thank Gavin Ball, Elahe Khalili and Paul Schwahn, who joined me in the project about spheres with minimal equators that I proposed during the Problem-Solving Workshop ``Computational Geometric Analysis" held at CUNY in June 2025. Many new insights on these metrics have come out our collaboration, and we hope to report on them together in the near future.

\section{The correspondence}\label{seccorrespondence}

\subsection{Algebraic curvature tensors and Killing symmetric tensors}\label{subsectcurvkilling}	Let $R:\mathbb{R}^{n+1}\times\mathbb{R}^{n+1}\times\mathbb{R}^{n+1}\times \mathbb{R}^{n+1}\rightarrow \mathbb{R}$ be an algebraic curvature tensor. This means that $R$ is linear in each of its four entries and satisfies the following identities:
	\begin{itemize}
		\item[$i)$] $R(x,y,z,w)= - R(y,x,z,w)$,
		\item[$ii)$] $R(x,y,z,w)=-R(x,y,w,z)$,
		\item[$iii)$] $R(x,y,z,w)=R(z,w,x,y)$,
		\item[$iv)$] (Bianchi identity) $R(x,y,z,w)+R(x,z,w,y)+R(x,w,y,z)=0$.
	\end{itemize}
	These are precisely the algebraic symmetries of the curvature tensor of a Riemannian metric $g$, defined as
	\begin{equation*}
		R_g(X,Y,Z,W) = g(\nabla^g_X\nabla^g_Y W - \nabla^g_Y\nabla^g_X W - \nabla^g_{[X,Y]} W,Z)
	\end{equation*}	
	for all tangent vector fields $X$, $Y$, $Z$, and $W$. (Beware different conventions!). In the formula above, $\nabla^g$ denotes the Levi-Civita connection of $g$.
	
	Let $Curv(\mathbb{R}^{n+1})$ be the set of algebraic curvature tensors in $\mathbb{R}^{n+1}$. It has the structure of a real vector space, and its dimension is $n(n+2)(n-1)^2/12$.	
	
	Every $R\in Curv(\mathbb{R}^{n+1})$ defines a function on the Grassmanian of unoriented two-planes in $\mathbb{R}^{n+1}$, known as the \textit{sectional curvature} function of $R$. Given $\sigma\subset\mathbb{R}^{n+1}$ a two-dimensional linear subspace, and chosen a basis $\{x,y\}\subset \sigma$, the formula
	\begin{equation*}
		sec_R(\sigma) = \frac{R(x,y,x,y)}{|x|^2|y|^2-\langle x,y \rangle^2}
	\end{equation*}
	computes the sectional curvature of $R$ at the plane $\sigma$. The symmetries of $R$ guarantee that the number $sec_R(\sigma)$ does not depend on the choice of basis of $\sigma$ used to compute it. Also, they guarantee that the tensor $R$ is uniquely determined by the knowledge of its sectional curvature function. In fact, one can write down an explicit (and quite long) polarization formula that computes $R(x,y,z,w)$ in terms of sectional curvatures of $R$ in appropriate planes, which are generated by linear combinations of the vectors $x$, $y$, $z$ and $w$. (See, for instance, \cite{Kar}).
	
	Given an algebraic curvature tensor $R$ on $\mathbb{R}^{n+1}$, let $k=k_R$ be the 2-tensor on $\mathbb{S}^n$ defined by
	\begin{equation}\label{equation.definitionkr}
		k_p(v,w) = R(p,v,p,w) \text{ for all } v,w\in T_p\mathbb{S}^n.
	\end{equation}
	Observe that $T_p\mathbb{S}^n$ is nothing else than the set of vectors of $\mathbb{R}^{n+1}$ that are orthogonal to $p$.
	
	The $2$-tensor $k$ has three remarkable properties, which are immediate consequences of the symmetries of $R$. First, $k$ is a symmetric two-tensor. Moreover, for every unit vector $v\in T_p\mathbb{S}^n$,
	\begin{equation*}
		k_p(v,v)=sec_R(span\{p,v\}).
	\end{equation*}
	Finally, given any normalised geodesic of $(\mathbb{S}^n,\overline{g})$,
	\begin{equation*}
		\gamma(t) = \cos(t)p+\sin(t)v,
	\end{equation*}
	where $\{p,v\}$ is an orthonormal basis for a two-dimensional linear space $\sigma\subset\mathbb{R}^{n+1}$,
	we have that
	\begin{equation*}
		k_{\gamma(t)}(\gamma'(t),\gamma'(t)) = sec_R(span\{\gamma(t),\gamma'(t)\})=sec_R(\sigma).
	\end{equation*}
	
	On any Riemannian manifold, the symmetric tensors that are constant along all of its geodesics are called \textit{Killing symmetric tensors}. In other words, the three observations above amount to the conclusion that the tensor $k_R$ is an element of the vector space $\mathcal{K}_2(\mathbb{S}^n,\overline{g})$ consisting of the Killing symmetric two-tensors of $(\mathbb{S}^n,\overline{g})$.
	
	\begin{lem}\label{lemacurvkil}
		The map
	\begin{equation*}
		R \in Curv(\mathbb{R}^{n+1}) \mapsto k_R \in \mathcal{K}_2(\mathbb{S}^n,\overline{g})
	\end{equation*}	
	defined by \eqref{equation.definitionkr} is a linear isomorphism.
	\end{lem}
	\begin{proof}
		The linearity is obvious. The injectivity follows from the fact that the only $R\in Curv(\mathbb{R}^{n+1})$ with $sec_R\equiv 0$ is $R=0$. As for surjectivity, Christiane Barbance \cite{Bar} computed the maximal dimension of the set of Killing symmetric two-tensors on a $n$-dimensional Riemannian manifold to be equal to $n(n+2)(n-1)^2/12=dim(Curv(\mathbb{R}^{n+1}))$. The result follows. 
	\end{proof}
	
	We remark that the formula for the dimension of $\mathcal{K}_2(\mathbb{S}^n,\overline{g})$ has been computed by several authors, and is sometimes referred to as the Delong-Takeuchi-Thompson dimension formula (after \cite{Del}, \cite{Tho} and \cite{Tak}).

	We will use Lemma \ref{lemacurvkil} as an intermediate step in the construction of the map from the set $Curv_{+}(\mathbb{R}^{n+1})$ of algebraic curvature tensors with \textit{positive sectional curvature} to the set $\mathcal{E}(\mathbb{S}^n)$ of Riemannian metrics on $\mathbb{S}^n$ with respect to which all equators are minimal. The key observation is that Lemma \ref{lemacurvkil} establishes a bijection between $Curv_{+}(\mathbb{R}^{n+1})$ and the set $\mathcal{K}^{+}_2(\mathbb{S}^n,\overline{g})$ consisting of the \textit{positive definite} Killing symmetric two-tensors of $(\mathbb{S}^n,\overline{g})$.
	
	Before we can proceed, we need to describe other general properties of Killing symmetric two-tensors. First of all, it will be useful to work with their tensorial characterisation. Let $(M,g)$ be a Riemannian manifold. Given a three-tensor $T$ on $M$, its \textit{cyclic symmetrisation} is the three-tensor $T^{S}$ on $M$ defined by
	\begin{equation*}
		T^S(X,Y,Z) = T(X,Y,Z) + T(Y,Z,X) + T(Z,X,Y) 
	\end{equation*}
	for all tangent vector fields $X$, $Y$ and $Z$ on $M$.
\begin{lem}
	Let $k$ be a symmetric two-tensor on a Riemannian manifold $(M,g)$. The following assertions are equivalent:
	\begin{itemize}
		\item[$i)$] $k$ is a Killing symmetric two-tensor of $(M,g)$.
		\item[$ii)$] $\nabla^g k(X,X,X)=0$ for all tangent vector fields $X$ on $M$.
		\item[$iii)$] $(\nabla^g k)^S=0$.
	\end{itemize}	
\end{lem}
\begin{proof}
	If $\gamma(t)$ is a geodesic of $(M,g)$ with $\gamma(0)=p$ and $\gamma'(0)=v$, then $(d/dt_{|t=0})k_{\gamma(t)}(\gamma(t),\gamma(t))=(\nabla^g k)_p(v,v,v)$. The equivalence between $i)$ and $ii)$ follows. The equivalence between $ii)$ and $iii)$ uses the symmetry of $k$, and follows by polarization.  
\end{proof}
	
	Specializing to the case of $(\mathbb{S}^n,\overline{g})$, it is possible to give a simple description of its Killing symmetric two-tensors. Given two Killing vector fields $K_1$ and $K_2$ of $(\mathbb{S}^n,\overline{g})$, the symmetric two-tensor defined by
	\begin{equation} \label{eqsymmetricproduct}
		(K_{1}\odot K_{2})(X,Y) = \overline{g}(K_{1},X)\overline{g}(K_{2},Y) + \overline{g}(K_{2},Y)\overline{g}(K_{1},X),
\end{equation}
for all tangent vector fields $X$ and $Y$, belongs to $\mathcal{K}_2(\mathbb{S}^n,\overline{g})$. Takeshi Sumitomo and Kwoichi Tandai showed that the elements of $\mathcal{K}_2(\mathbb{S}^n,\overline{g})$ thus constructed form a set of generators. As each Killing vector field $K$ of $(\mathbb{S}^n,\overline{g})$ corresponds to a unique skew-symmetric matrix $V\in \mathfrak{so}(n+1,\mathbb{R})$ so that $K(p)=Vp$ for every $p\in \mathbb{S}^n$, we see that a fairly concrete and useful representation of all elements of the space $\mathcal{K}_2(\mathbb{S}^n,\overline{g})$ is available. 

\subsection{Metrics with minimal equators and positive definite Killing symmetric two-tensors} We first deduce the formula for the mean curvature of equators passing on a given non-empty open subset $U\subset \mathbb{S}^n$ with an arbitrary Riemannian metric $g$, and then analyse what happens when all of them are minimal in $(U,g)$.

	Given a non-zero vector $v$ in $\mathbb{R}^{n+1}$, consider the auxiliary function defined by
\begin{equation*}
	 \hat{V} : p\in \mathbb{S}^n \mapsto \langle p,v \rangle \in \mathbb{R},
\end{equation*}
where we use brackets to denote the Euclidean inner product on $\mathbb{R}^{n+1}$, so that $\overline{g} = \langle -,-\rangle_{\rvert_{\mathbb{S}^n}}$. Then, $\hat{V}$ is a smooth function on $\mathbb{S}^n$ that has zero as a regular value. The set
\begin{equation*}
	\Sigma_{v}=\hat{V}^{-1}(0)=\{x\in \mathbb{S}^n\,|\,\langle x,v\rangle = 0\}
\end{equation*}
is just the equator in $\mathbb{S}^n$ that is orthogonal to $v$. Clearly, $\Sigma_{-v}=\Sigma_v$.
	
	Let $g$ be an arbitrary Riemannian metric on an open subset $U$ of $\mathbb{S}^n$. Then
\begin{equation*}
	N^g = \frac{\nabla^{g} \hat{V}}{|\nabla^{g} \hat{V}|_g}
\end{equation*} 
defines a unit normal vector field on $\Sigma_v\cap U$. The second fundamental form of $\Sigma_v\cap U$ at a point $p\in \Sigma_v\cap U$ is then given by
\begin{equation*}
	A_g(X,Y) = g(\nabla^{g}_{X}{N^g},Y) = \frac{1}{|\nabla^{g} \hat{V}|_g} Hess_{g} \hat{V}(X,Y)
\end{equation*}
for all $X$, $Y\in T_{p}\Sigma_v$. Taking the trace over $\Sigma_v$, we have:
\begin{lem}
	The mean curvature of $\Sigma_v\cap U$ with respect to $g$ is
\begin{equation} \label{eqmeancurvature}
	H_ g = \frac{1}{|\nabla^{g}\hat{V}|_g}\left(\Delta_{g}\hat{V} - Hess_{g}\hat{V}(N^g,N^g)\right).
\end{equation}
\end{lem}

	Recall that the Hessian of the function $\hat{V}$ with respect to a Riemannian metric $g$ is computed by
\begin{equation} \label{eqdefhessian}
	Hess_{g}\hat{V}(X,Y) = X(Y\hat{V}) - (\nabla^{g}_{X}Y)\hat{V}
\end{equation}
for all tangent vector fields $X$ and $Y$. In the particular case of the canonical metric $\overline{g}$, since the function $\hat{V}$ satisfies Obata's equation,
\begin{equation*}
	Hess_{\overline{g}}\hat{V}+ \hat{V}\overline{g}=0,
\end{equation*}
$Hess_{\overline{g}}\hat{V}$ vanishes identically on $\Sigma_v=\hat{V}^{-1}(0)$. Combining this fact with \eqref{eqdefhessian}, we deduce the following identity, which is valid for an arbitrary Riemmanian metric $g$ on an open subset $U$ of $\mathbb{S}^n$:
\begin{multline}
	(Hess_{g}\hat{V})_{p}(X,Y) =  (d\hat{V})_{p}(\overline{\nabla}_{X}Y - \nabla^{g}_{X}Y) \label{eqhessianconection} \\
	\text{for all }p\in \Sigma_{v}\cap U, \text{ and for all } X,Y\in T_{p}\mathbb{S}^n.
\end{multline}

	Consider the three-tensor $\mathcal{T}_g$ on $U$ given by
	\begin{equation} \label{eqfundamentaltensor}
			\mathcal{T}_{g}(X,Y,Z) = g(\nabla^{g}_{X}Y - \overline{\nabla}_{X}Y,Z)
	\end{equation} 
	for all tangent vector fields $X$, $Y$ and $Z$.

\begin{lem} \label{lemfundamentaltensor}
	Let $g$ be a Riemannian metric on an open subset $U\subset \mathbb{S}^n$. The three-tensor $\mathcal{T}_g$ defined in \eqref{eqfundamentaltensor} satisfies the following properties:
	\begin{itemize}
		\item[$1)$] $\mathcal{T}_g$ is symmetric in the first two entries.
		\item[$2)$] The cyclic symmetrisation of $\mathcal{T}_g$ is
		\begin{equation*}
			(\mathcal{T}_g)^{S} = \frac{1}{2} (\overline{\nabla}g)^S.
		\end{equation*}
		\item[$3)$] The trace of $\mathcal{T}_{g}$ with respect to $g$ in the second and third entries is
			\begin{equation*}
				tr^{23}_{g}(\mathcal{T}_g) = d\log(\psi),
			\end{equation*}	
			where $\psi$ is the positive smooth function on $U$, uniquely defined by the volume elements of $(U,g)$ and $(U,\overline{g})$, that satisfies
			\begin{equation*}
				dV_g = \psi dV_{\overline{g}}.
			\end{equation*}
		\item[$4)$] For every $v\in \mathbb{R}^{n+1}\setminus\{0\}$, $p\in \Sigma_v\cap U$, and $X,Y\in T_{p}\mathbb{S}^n$,
			\begin{equation*}
				\mathcal{T}_{g}(X,Y,\nabla^{g}\hat{V}(p)) = - Hess_{g}\hat{V}(X,Y),
			\end{equation*}					
			and the trace of $\mathcal{T}_{g}$ with respect to $g$ in the first and second entries satisfies, at every point $p\in \Sigma_v\cap U$, 
			\begin{equation*}
				tr^{12}_{g}(\mathcal{T}_g)(\nabla^g \hat{V}(p)) = -\Delta_g \hat{V}(p).
			\end{equation*}		
	\end{itemize}
\end{lem}
\begin{proof}
	Property $1)$ is a consequence of the fact that both connections $\nabla^g$ and $\overline{\nabla}$ are torsion free.
	
	Using the compatibility of the connection $\nabla^g$, we also have
	\begin{align*}
		(\overline{\nabla} g) (X,Y,Z)& = Zg(X,Y)-g(\overline{\nabla}_Z X,Y)-g(X,\overline{\nabla}_Z Y)\\
		& = g(\nabla^g_Z X,Y) + g(X,\nabla^g_Z Y) -g(\overline{\nabla}_Z X,Y)-g(X,\overline{\nabla}_Z Y) \\
		& = \mathcal{T}_g(Z,X,Y)+\mathcal{T}_g(Z,Y,X) \\
		& = \mathcal{T}_g(Z,X,Y)+\mathcal{T}_g(Y,Z,X),
	\end{align*}
	by Property $1)$. Property $2)$ follows.
	
	Let $\{e_i\}$ be a local $g$-orthonormal frame. By definition, 
	\begin{equation*}
		dV_g(e_1,\ldots,e_n)=1 \quad \text{and} \quad dV_{\overline{g}}(e_1,\ldots,e_n)=\psi^{-1}
	\end{equation*}
	on the domain of the frame. 	The volume forms $dV_g$ and $dV_{\overline{g}}$ are parallel with respect to the Levi-Civita connections of respective metrics. Hence, for every tangent vector field $X$ on $U$,
	\begin{align*}
		tr_{g}^{23}(\mathcal{T}_g)(X) & = \sum_{i=1}^n g(\nabla^g_X e_i-\overline{\nabla}_X e_i,e_i) \\
		& =\sum_{i=1}^{n} dV_g(e_1,\ldots,(\nabla^g_X e_i-\overline{\nabla}_X e_i),\ldots,e_n)\\
		& = X(dV_g(e_1,\ldots,e_n))-\sum_{i=1}^{n} \psi dV_{\overline{g}}(e_1,\ldots,\overline{\nabla}_X e_i,\ldots e_n) \\
		& = -\psi X(dV_{\overline{g}}(e_1,\ldots,e_n))=-\psi X(\psi^{-1})=d(\log(\psi))(X).
	\end{align*}
	Property $3)$ follows.
	
	Finally, Property $4)$ is a direct consequence of identity \eqref{eqhessianconection}.
\end{proof}

	We can now express the condition that all equators passing through $U$ are minimal in $(U,g)$ as an equation involving $g$ and $\mathcal{T}_g$, or even as an equation involving $g$ and the function $\psi=dV_g/dV_{\overline{g}}$. 
\begin{prop}  \label{propfundamentalequation}
Let $g$ be a Riemannian metric on an open subset $U$ of the sphere $\mathbb{S}^n$. The following three statements are equivalent:
\begin{itemize}
\item[$i)$] All equators intersecting $U$ are minimal.
\item[$ii)$] The three-tensor $\mathcal{T}_g$ defined in \eqref{eqfundamentaltensor} satisfies the equation
	\begin{equation} \label{eqtensorequation1}
		\left(\mathcal{T}_g - g\otimes tr^{12}_{g}(\mathcal{T}_g)\right)^{S}=0 \quad \text{on} \quad U.
	\end{equation}
\item[$iii)$] The metric $g$ and the smooth positive function $\psi$ on $U$ uniquely defined by $dV_{g}=\psi dV_{\overline{g}}$ satisfy
	\begin{equation} \label{eqmetricequation1}
		\left(\overline{\nabla} g - \frac{4}{n+1}d\log(\psi)\otimes g\right)^S = 0 \quad \text{on} \quad U.
	\end{equation}
\end{itemize}	
\end{prop}
\begin{proof}
	A point $p$ in $U$ belongs to the equator $\Sigma_{v}$ if and only if $v\in \mathbb{R}^{n+1}\setminus\{0\}$ is orthogonal to $p$, \textit{i.e.} if and only if $v\neq 0$ belongs to $T_{p}\mathbb{S}^n$. Hence, given an arbitrary point $p$ in $U$, it follows from equation \eqref{eqmeancurvature} and Lemma \ref{lemfundamentaltensor}, item $4)$, that the equality
\begin{equation*}
	\mathcal{T}_{g}(\nabla^{g}\hat{V}(p),\nabla^{g}\hat{V}(p),\nabla^{g}\hat{V}(p)) = |\nabla^{g}\hat{V}(p)|^2_{g}\,tr^{12}_{g}(\mathcal{T}_g)(\nabla^{g}\hat{V}(p))
\end{equation*}
holds for every $v\neq 0$ in $T_{p}\mathbb{S}^{n}$ if, and only if, $g$ is a Riemannian metric in $U$ with respect to which all equators containing $p$ have zero mean curvature at $p$.

	The linear map $v \in T_{p}\mathbb{S}^n \mapsto \nabla^{g}\hat{V}(p) \in T_{p}\mathbb{S}^{n}$ is an isomorphism (otherwise, $0$ would be a critical value of $\hat{V}$ for some $v\neq 0$ in $T_{p}\mathbb{S}^n$). Hence, and since $p\in U$ is also arbitrary, the condition
\begin{equation} \label{eqtensorequationa}
		\mathcal{T}_{g}(X,X,X) = g(X,X)tr^{12}_{g}(\mathcal{T}_{g})(X)
\end{equation}
for every tangent vector field $X$ on $U$ is necessary and sufficient for a Riemannian metric $g$ on $U$ to be such that all equators are minimal in $(U,g)$.
	
	The equivalence between $i)$, $ii)$ and $iii)$ will be a consequence of the following observations. First, by polarization, the identity \eqref{eqtensorequationa} is equivalent to its symmetrised version,
\begin{equation} \label{eqtensorequation2}
	\left(\mathcal{T}_g - g\otimes tr^{12}_{g}(\mathcal{T}_g)\right)^{S}=0 \quad \text{on} \quad U,
\end{equation} 
because both $g\otimes tr^{12}_{g}(\mathcal{T}_g)$ and $\mathcal{T}_{g}$ are symmetric in the first two entries. This shows that $i)$ is equivalent to $ii)$.

	Moreover, observe that, in general, a three-tensor $\mathcal{T}$ satisfies
\begin{equation}\label{eqtracegT}
	tr_{g}^{12}( (g\otimes tr_{g}^{12}(\mathcal{T}))^S) = (n+2)tr_{g}^{12}(\mathcal{T}).
\end{equation}
	Another general formula that holds whenever a three-tensor $\mathcal{T}$ is symmetric in the first two variables is
\begin{equation*}
	tr_{g}^{12}(\mathcal{T}^S) = tr_{g}^{12}(\mathcal{T}) + 2tr_g^{23}(\mathcal{T}).
\end{equation*}

	Using the last two general identities together with Properties $2)$ and $3)$ of $\mathcal{T}_g$, we deduce that the tensor $\mathcal{S}_g=\mathcal{T}_g-g\otimes tr_g^{12}(\mathcal{T}_g)$ satisfies
\begin{equation*}
	\mathcal{S}_g^{S} = \left(\frac{1}{2}\overline{\nabla}g - \frac{2}{n+1}g\otimes d\log(\psi) + \frac{1}{n+1}g\otimes tr_{g}^{12}(\mathcal{S}_g^S)\right)^{S}
\end{equation*}
	It is now straightforward to check, using \eqref{eqtracegT}, that 
	\begin{equation*}
		\mathcal{S}_g^{S}=0 \text{ on $U$} \,\, \Leftrightarrow \,\, \overline{\nabla}g - \frac{4}{n+1} g\otimes d\log(\psi)=0 \text{ on $U$}.
	\end{equation*} 
	Hence $ii)$ and $iii)$ are equivalent, as claimed.
\end{proof}

	We can prove the most important ingredient of the first part of the Classification Theorem, which we formulate in a slightly more general set-up:	 	
	
\begin{thm} \label{thmcorrespondenceAfull}
	Let $U$ be a non-empty open subset of $\mathbb{S}^n$, $n\geq 2$. 
	\begin{itemize}
	\item[$i)$] Let $g$ be a Riemannian metric on $U$ and denote by $F_g$ the positive smooth function on $U$ uniquely defined by the identity 
	\begin{equation*}
		dV_g = F_g^{(n+1)/4} dV_{\overline{g}}.
	\end{equation*} 
	If the intersection of all equators with $U$ are minimal hypersurfaces in $(U,g)$, then the tensor
	\begin{equation*}
		k_g := \frac{1}{F_g} g
	\end{equation*}		
	is a positive definite Killing symmetric two-tensor of $(U,\overline{g})$.
	\item[$ii)$] Let $k$ be a positive definite Killing symmetric two-tensor of $(U,\overline{g})$ and denote by $D_k$ the positive smooth function on $U$ uniquely defined by
	\begin{equation*}
		dV_k = D_k^{(n-1)/4} dV_{\overline{g}}.
	\end{equation*} 
	Then, the tensor
	\begin{equation*}
		g_k := \frac{1}{D_k} k
	\end{equation*}	
	is a Riemannian metric on $U$ with respect to which the intersection of all equators with $U$ are minimal hypersurfaces in $(U,g)$.
	\item[$iii)$] Under the notations of items $i)$ and $ii)$, the maps
	\begin{equation*}
		g \mapsto k_g \quad \text{and} \quad k \mapsto g_k
	\end{equation*}
	are the inverse maps of each other.
\end{itemize} 
\end{thm}
\begin{proof}
 \textit{Item i)} Using the notation of Proposition \ref{propfundamentalequation}, we have $\psi = F_{g}^{(n+1)/4}$. Hence, the symmetric two-tensor $k_g = (1/F_g)g$, which is clearly positive definite, satisfies
\begin{equation*}
	\overline{\nabla}k_g = \frac{1}{F_g}\overline{\nabla}g -\frac{1}{F_{g}^2}(g\otimes dF_{g}) =\frac{1}{F_g}\left( \overline{\nabla}g - \frac{4}{n+1}g\otimes d\log(\psi)\right).
\end{equation*}
By Proposition \ref{propfundamentalequation}, item $iii)$, $k_g = (1/F_g)g$ is a positive definite Killing symmetric two-tensor of $(U,\overline{g})$.

	\textit{Item ii)} The symmetric two-tensor $g_k=(1/D_k)k$ is clearly a Riemannian metric on $U$, which satisfies
\begin{equation} \label{eqvolumeform}
	dV_{g_k} = \frac{1}{D_k^{n/2}}dV_{k} = \frac{D_k^{(n-1)/4}}{D_{k}^{n/2}} dV_{\overline{g}} = \frac{1}{D_{k}^{(n+1)/4}}dV_{\overline{g}},
\end{equation}
by the definition of $D_k$. Thus, using the notation of Proposition \ref{propfundamentalequation}, $\psi = D_{k}^{-(n+1)/4}$. Since $k$ is a Killing symmetric two-tensor on $(U,\overline{g})$, 
\begin{equation*}
	(\overline{\nabla}g_k)^{S} = \frac{1}{D_k}(\overline{\nabla}k)^S-\frac{1}{D_k^2} \left(k\otimes dD_k\right)^S = \frac{4}{n+1}(g_k \otimes d\log(\psi))^S.
\end{equation*}
By Proposition \ref{propfundamentalequation}, item $iii)$, $g_k$ is a Riemannian metric on $U$ with respect to which the intersection of all equators with $U$ are minimal hypersurfaces.

	\textit{Item iii)} Using the string of equations \eqref{eqvolumeform} and the definitions of $F_g$ and $D_k$, we check that $F_{g_k}D_{k}=1$. Similarly, $F_gD_{k_g}=1$ holds. Since $g=F_gk_g$ and $g_k = (1/D_k)k$, the assertion follows.
\end{proof}

\begin{rmk} \label{rmkkne}
	Morris Knebelman \cite{Kne} showed that if two Riemannian manifolds are related by a diffeomorphism that maps (unparametrised) geodesics into (unparametrised) geodesics, then there exists a bijection between the set of Killing symmetric two-tensors of the two manifolds. This result clarifies that there is relation between Hangan's correspondence  (Theorem 2 in \cite{Han1}) and the correspondence described in Theorem \ref{thmcorrespondenceAfull}. In fact, the \textit{central} or \textit{gnomonic projections} from open hemispheres of $\mathbb{S}^n$ into $\mathbb{R}^n$ are examples of such maps.
\end{rmk}

\begin{rmk}\label{rmkmaximaldomain}
	Let $U$ be a non-empty connected open subset of $\mathbb{S}^n$. Every Killing symmetric two-tensor of $(U,\overline{g})$ has a unique extension to an element of $\mathcal{K}_2(\mathbb{S}^n,\overline{g})$, because the restriction map to $U$ is an injective linear map and has a codomain of dimension at most equal to the dimension of $\mathcal{K}_2(\mathbb{S}^n,\overline{g})$, as shown in \cite{Bar}. Each Killing symmetric two-tensor of $(\mathbb{S}^n,\overline{g})$ has, therefore, a (possibly empty) maximal domain where it is positive definite. In light of Theorem \ref{thmcorrespondenceAfull}, it will be interesting to understand the structure of the connected components of these maximal domains. (See, for instance, Corollary \ref{corantipodalsymmetry} and Remark \ref{rmkrestrictequivariance}).
\end{rmk}

\subsection{The algorithm}\label{subsectalgorithm} It is useful to notice that the above proof of Theorem \ref{thmcorrespondenceAfull} is algorithmic in nature, in the sense that it constructs an explicit map from the space of Riemannian metrics on $U\subset \mathbb{S}^n$ with minimal equators to the space of positive Killing symmetric two-tensors on $(U,\overline{g})$, and at the same time constructs its explicit inverse map. The same can be said about the proof of Lemma \ref{lemacurvkil}. This means that we can formulate the first part of the Classification Theorem as follows:
	
	\begin{thm} \label{thmcorrespondenceasalgorithm}
	Let $\mathcal{E}(\mathbb{S}^n)$ be the set of Riemannian metrics on $\mathbb{S}^n$ with minimal equators, and let $Curv_+(\mathbb{R}^{n+1})$ be the set of algebraic curvature tensors on $\mathbb{R}^{n+1}$ with positive sectional curvature. \\
	
	Consider the following two algorithms: \\

	\textit{1. Input:} $R\in Curv_{+}(\mathbb{R}^{n+1})$.
	\begin{enumerate}
		\item compute the positive definite symmetric two-tensor $k_R$ on $\mathbb{S}^n$ defined by 
		\begin{equation*}
			(k_R)_p(v,w) = R(p,v,p,w)
		\end{equation*}
		for all $p\in \mathbb{S}^n$ and for all vectors $v$, $w\in T_p\mathbb{S}^n$.
		\item compute the positive smooth function $D_R$ on $\mathbb{S}^n$ such that
		\begin{equation*}
			dV_{k_R} = D_R^{\frac{n-1}{4}}dV_{\overline{g}},
		\end{equation*}
		where $dV_{k_R}$ is the volume element of $k_R$ on $\mathbb{S}^n$.
		\item compute the positive definite symmetric two-tensor
		\begin{equation*}
			g_R = \frac{1}{D_R} k_R. 
		\end{equation*}
	\end{enumerate}
	
	\textit{Output:} $g_R\in \mathcal{E}(\mathbb{S}^n)$.	
	
	\vspace{4mm}
	
	\noindent and

		\vspace{4mm}

\textit{2. Input:} $g\in \mathcal{E}(\mathbb{S}^n)$.
	\begin{enumerate}
		\item compute the positive smooth function $F_g$ on $\mathbb{S}^n$ such that
		\begin{equation*}
			dV_g = F_g^{\frac{n+1}{4}}dV_{\overline{g}},
		\end{equation*}
		where $dV_g$ is the volume element of $g$ on $\mathbb{S}^n$. 
		\item compute the Riemannian metric
		\begin{equation*}
			k_g = \frac{1}{F_g} g.
		\end{equation*}
		\item compute the algebraic curvature tensor $R_g$ such that
		\begin{equation*}
			R_g(p,v,p,v) = (k_g)_p(v,v) 
		\end{equation*}
		for every $p\in \mathbb{S}^n$ and every vector $v\in T_p\mathbb{S}^n$.
	\end{enumerate}
	
	\textit{Output:} $R_g\in Curv_+(\mathbb{R}^{n+1})$.

	\vspace{4mm}

\noindent Then, the maps
	\begin{equation*}
		g\in \mathcal{E}(\mathbb{S}^n) \mapsto R_g \in Curv_+(\mathbb{R}^{n+1})	\text{ and } R \in Curv_+(\mathbb{R}^{n+1}) \mapsto g_R\in \mathcal{E}(\mathbb{S}^n)
	\end{equation*}
are well-defined and are the inverse of each other.
\end{thm}

\begin{proof}
	Combine Lemma \ref{lemacurvkil} and Theorem \ref{thmcorrespondenceAfull} (for $U=\mathbb{S}^n$).
\end{proof}

\begin{rmk}
	From the explicit formula for the correspondence, it is possible to check that the metrics on $\mathbb{S}^n$ with respect to which all equators are minimal are actually analytic.
\end{rmk}

	Let us explore the algorithm by computing some simple examples.
	
\begin{exam}
	The correspondence clearly maps the canonical metric $\overline{g}$ into the unique algebraic curvature tensor with constant sectional curvature one, namely, $R(x,y,z,w)=\langle x,z\rangle\langle y,w\rangle-\langle x,w\rangle\langle x,z\rangle$. It also sends scalings of $R$ into scalings of $g$ (but the scaling factor is affected by a dimensional power).
\end{exam}

\begin{exam}	
	The next simplest algebraic curvature tensor coming from a compact symmetric space is the curvature tensor $R$ of $\mathbb{CP}^m$, $m\geq 2$. Let $J:\mathbb{R}^{2m+2}\rightarrow\mathbb{R}^{2m+2}$ corresponds to multiplication by the imaginary number $i$ under the natural identification 
	\begin{equation*}
		(x_1,y_1,\ldots,x_{m+1},y_{m+1}) \in \mathbb{R}^{2m+2}\sim (x_1+iy_1,\ldots,x_{m+1}+iy_{m+1})\in \mathbb{C}^{m+1}.
	\end{equation*}
	Then 
	\begin{align*}
		R(x,y,z,w) = & \langle x,z\rangle \langle y,w\rangle - \langle x,w\rangle \langle y,z\rangle \\ & + \langle Jx,z\rangle \langle Jy, w\rangle - \langle Jx,w\rangle \langle Jy,z\rangle + 2\langle Jx,y\rangle \langle Jz,w\rangle.
	\end{align*}
	For all $p\in \mathbb{S}^{2m+1}\subset \mathbb{R}^{2m+2}$ and all $v\in T_p\mathbb{S}^{2m+1}$, we have
	\begin{equation*}
		(k_R)_p(v,w) = \langle v,w\rangle + 3\langle Jp,v\rangle\langle Jp,w\rangle.
	\end{equation*}
	Notice that the sectional curvatures of $R$ can be any number in the interval $[1,4]$, as it is well-known.
	
	Choose $\{e_1=Jp,e_2,\ldots,e_{2m+1}\}\subset T_p\mathbb{S}^{2m+1}$ a $\overline{g}$-orthonormal basis. Notice that $(Jp)^{\perp}\cap T_{p}\mathbb{S}^{2m+1}$ is invariant under $J$. In this basis, $[(k_R)_p(e_i,e_j)]$ is a diagonal matrix with two eigenvalues, namely, $4$ with multiplicity one and $1$ with multiplicity $2m$. Therefore
	\begin{equation*}
		D_{k_R}(p) = \left(\sqrt{det[(k_R)_p(e_i,e_j)]}\right)^{\frac{4}{2m}} = 4^{\frac{1}{m}},
	\end{equation*}
	so that
	\begin{equation*}
		(g_{R})_p(X,Y) = \frac{1}{4^{\frac{1}{m}}}(\overline{g}(X,Y) + 3\overline{g}(Jp,X)\overline{g}(Jp,Y)).
	\end{equation*}
	Up to the scaling factor $4^{-1/m}$, we recognize this as a Berger deformation of the round metric, which scale the metric $\overline{g}$ just in the direction of the vector field that generates the Hopf action (that is, the vector field $p\mapsto Jp$) so that it has constant length equal to $2$.
	 
	We remark the peculiarity of the dimension $2m=4$: in this case, the scalar curvature of $g_R$ is zero! However, this is not the case in higher dimensions (\textit{Cf}. \cite{BetPic}, Proposition 4.2).
\end{exam}

\begin{exam}\label{exampleleftinvar}
	Let $g$ be a left-invariant metric on $\mathbb{S}^3$, the Lie group of unit quaternions. Let $L_i$, $L_j$ and $L_k$ denote the left-invariant vectors fields on $\mathbb{S}^3$ extending the imaginary quaternions $i$, $j$, $k\in T_1\mathbb{S}^3$. These vectors form an orthonormal frame Killing vector fields of $(\mathbb{S}^3,\overline{g})$, and we can write
	\begin{equation*}
		g = a\, L_i\odot L_i + b\, L_{j}\odot L_j + c\, L_k\odot L_k
	\end{equation*}
	for some positive constants $a$, $b$ and $c$. 
	
	It follows that $g$ is a positive definite Killing symmetric two-tensor itself. Since $dV_g/dV_{\overline{g}}=\sqrt{abc}$ is constant, the same can be said about the tensor $k_g=(1/F_g)g$. By Theorem \ref{propfundamentalequation}, $g=g_{k_g}$ is a metric with minimal equators.	
\end{exam}

\section{The actions} \label{secactions}

	Let $\tilde{\mathcal{G}}_n$ be the group of diffeomorphisms of $\mathbb{S}^n$ that map equators into equators. This group acts on the space $\mathcal{E}(\mathbb{S}^n)$ by pull-back, because, for every diffeomorphism $\phi$ of $\mathbb{S}^n$, a submanifold $\Sigma\subset \mathbb{S}^n$ is minimal with respect to a pulled back metric $\phi^*g$ if and only if the submanifold $\phi(\Sigma)$ is minimal with respect to $g$. In order to formulate questions about the moduli space of metrics on $\mathbb{S}^n$ with minimal equators, we need first to understand this group and this action.

\subsection{The $GL(n+1,\mathbb{R})$-action on $\mathcal{E}(\mathbb{S}^n)$}	
	Generalising the terminology we have been using so far, we define a \textit{$k$-equator} as the intersection of $\mathbb{S}^n$ with some $(k+1)$-dimensional linear subspace of $\mathbb{R}^{n+1}$, for every $k=0,\ldots,n-1$. Given an element of $T\in GL(n+1,\mathbb{R})$, that is, a linear invertible map $T:\mathbb{R}^{n+1}\rightarrow\mathbb{R}^{n+1}$, the map
\begin{equation}\label{eqphiT}
	\phi(T) : x\in \mathbb{S}^n \mapsto \frac{T(x)}{|T(x)|}\in \mathbb{S}^n
\end{equation}
defines an element of $\tilde{\mathcal{G}}_n$. 

	The map
\begin{equation} \label{eqhomomorphism2}
	T \in GL(n+1,\mathbb{R}) \mapsto \phi(T) \in \tilde{\mathcal{G}}_n
\end{equation} 
is a group homomorphism. The next proposition, which is a consequence of a classical algebraic result known as the \textit{Fundamental Theorem of Projective Geometry} (\textit{cf}. \cite{Art}, Chapter II, 10), shows that \eqref{eqhomomorphism2} is a surjective homomorphism.
\begin{prop} \label{propnaturalsymmetries}
	Let $\tilde{\mathcal{G}}_n$ be the group consisting of diffeomorphisms of the sphere $\mathbb{S}^n$ that permute $(n-1)$-equators. Every element of $\tilde{\mathcal{G}_n}$ is of the form $\phi(T)$ for some $T\in GL(n+1,\mathbb{R})$ as in \eqref{eqphiT}.
\end{prop}
\begin{proof}
	For every non-negative integer $k<n$, any $k$-equator of $\mathbb{S}^n$ is the intersection of $n-k$ different $(n-1)$-equators in general position. Conversely, the intersection of $n-k$ different $(n-1)$-equators in general position defines a unique $k$-equator.  Hence, every diffeomorphim $\phi$ in $\tilde{\mathcal{G}}_n$ is a map that permutes $k$-equators, for every $k=0,\ldots,n-1$.
	
	It follows that the map that assigns to each proper vector subspace $V\subset \mathbb{R}^{n+1}$ the proper vector subspace generated by the $(dim(V)-1)$-equator $\phi(V\cap \mathbb{S}^n)$ is a collineation, in the sense that it permutes proper vector subspaces of $\mathbb{R}^{n+1}$ while preserving the partial order induced by inclusion. By the Fundamental Theorem of Projective Geometry, there exists $T\in GL(n+1,\mathbb{R})$ such that, for every proper vector subspace $V$ of $\mathbb{R}^{n+1}$, the subspace $T(V)$ is precisely the subspace generated by $\phi(V\cap \mathbb{S}^n)$.
	
	Specialising to one-dimensional subspaces, we conclude that, for every $x$ in $\mathbb{S}^n$, the point $\phi(x)$ in $\mathbb{S}^n$ must be either equal to $T(x)/|T(x)|$ or to $-T(x)/|T(x)|$. Since $T$ is a linear isomorphism, the continuous map $x\in \mathbb{S}^n \mapsto \langle T(x),\phi(x) \rangle \in \mathbb{R}$ never vanishes. Thus, either $\phi=\phi(T)$ or $\phi=\phi(-T)$ as in \eqref{eqphiT}, as we wanted to prove.
\end{proof}

	The kernel of the homomorphism \eqref{eqhomomorphism2}, on the other hand, is easier to determine. If $\phi(T)(x)=x$ for every $x\in \mathbb{R}^{n+1}$, then $T$ maps each unit vector into a positive multiple of itself. But the multiple needs to be the same. For, otherwise,  in the plane generated by two unit vectors $v$ and $w$ corresponding, to different eigenvalues of $T$ one would see $\phi(T)$ mapping $(v+w)/\sqrt{2}$ into a vector that is not a positive multiple of itself. Thus, the kernel of the homomorphism \eqref{eqhomomorphism2} consists of the maps $\lambda Id$, where $\lambda>0$. 
	
	Collecting these facts together, we conclude that the map \eqref{eqhomomorphism2} induces an isomorphism 
\begin{equation*}
	\tilde{\mathcal{G}_n} \simeq GL(n+1,\mathbb{R})/\{\lambda Id\,|\,\lambda >0\},
\end{equation*}
and that the group $GL(n+1,\mathbb{R})$ acts on $\mathcal{E}(\mathbb{S}^n)$ by the rule
\begin{equation}\label{eqactionpullcack}
	g\cdot T = \phi(T)^{*}g.
\end{equation}

	It is interesting to remark that $\phi(-T)$ and $\phi(T)$ induce the same permutation in the set of equators. At this point, however, it is not clear whether $T$ and $-T$ act on $\mathcal{E}(\mathbb{S}^n)$ in the same way or not. This question will be settled at the end of this Section (see Corollary \ref{corantipodalsymmetry}). 

\subsection{The $GL(n+1,\mathbb{R})$-action on $\mathcal{K}_2(\mathbb{S}^n,\overline{g})$} \label{subsecactionkilling} Let $T$ be an element of the general linear group $GL(n+1,\mathbb{R})$ and denote by $T^\top$ its transpose. If $V\in \mathfrak{so}(n+1,\mathbb{R})$ is a skew-symmetric matrix, it is immediate to check that $T^{\top}VT$ is also a skew-symmetric matrix. Taking the identification between Killing vector fields of $(\mathbb{S}^n,\overline{g})$ and skew-symmetric matrices into account, we can define the right action of $T\in GL(n+1,\mathbb{R})$ on a Killing vector field $K$ of $(\mathbb{S}^n,\overline{g})$ by the rule
\begin{equation*}
K \cdot T = \frac{1}{det(T)^{\frac{2}{n+1}}}T^{\top}KT.
\end{equation*} 

	This linear action extends naturally to the spaces of Killing symmetric tensor fields. It is enough to define it on a generating set, for instance, on the set of elements of the form $K\odot L \in \mathcal{K}_{2}(\mathbb{S}^n,\overline{g})$ as in \eqref{eqsymmetricproduct}, where $K$ and $L$ are Killing vector fields of $(\mathbb{S}^n,\overline{g})$. Then
	\begin{equation} \label{eqactionkilling}
		(K\odot L) \cdot T = \frac{1}{det(T)^{\frac{4}{n+1}}} (K \cdot T)\odot (L \cdot T),
	\end{equation} 
	Notice that the subset $K_2^{+}(\mathbb{S}^n,\overline{g})$ is invariant by this action.
		
	It is straightforward to check that, for a linear orthogonal transformation $T\in O(n+1,\mathbb{R})\subset GL(n+1,\mathbb{R})$, the action described above coincides with the standard action of $T$, regarded as an isometry of $(\mathbb{S}^n,\overline{g})$, by the pull-back operation on the space $\mathcal{K}_2(\mathbb{S}^n,\overline{g})$.

\subsection{The $GL(n+1,\mathbb{R})$-action on $Curv(\mathbb{R}^{n+1})$} Let $T$ be an element of the general linear group $GL(n+1,\mathbb{R})$. Given $R\in Curv(\mathbb{R}^{n+1})$, we define the four-linear map $R\cdot T$ on $\mathbb{R}^{n+1}$ by
\begin{equation}\label{eqactioncurvature}
	R\cdot T(x,y,z,w) = \frac{1}{det(T)^{\frac{4}{n+1}}}R(Tx,Ty,Tz,Tw).
\end{equation}
	It is immediate to check that this is a well-defined operation that defines an action of $GL(n+1,\mathbb{R})$ on $Curv(\mathbb{R}^{n+1})$ that extends the standard action of $O(n+1,\mathbb{R})$. Also, it is clear that the subset $Curv_{+}(\mathbb{R}^{n+1})$ is preserved by this action.

\subsection{Equivariance} \label{subsecequivar} We show next that the $GL(n+1,\mathbb{R})$-actions on $\mathcal{E}(\mathbb{S}^n)$, $\mathcal{K}_2^+(\mathbb{S}^n,\overline{g})$ and $Curv_+(\mathbb{R}^{n+1})$ described in the previous subsections (see \eqref{eqactionpullcack}, \eqref{eqactionkilling} and \eqref{eqactioncurvature}) are intertwined by the maps between these spaces defined in Section \ref{seccorrespondence}. 
	
	The map $R\in Curv(\mathbb{R}^{n+1})\mapsto k_R\in \mathcal{K}_2(\mathbb{S}^n,\overline{g})$ defined in Lemma \ref{lemacurvkil} certainly enjoys this property. Thus, the key fact to be proven is the following statement:

\begin{thm} \label{thmequivariance}
	The correspondence between Riemannian metrics on $\mathbb{S}^n$ with respect to which all equators are minimal and positive definite Killing symmetric two-tensors of $(\mathbb{S}^n,\overline{g})$ described in Theorem \ref{thmcorrespondenceAfull} is equivariant with respect to the $GL(n+1,\mathbb{R})$-actions in both spaces.
\end{thm}
\begin{proof}
	Let $T\in GL(n+1,\mathbb{R})$, and set $\phi:=\phi(T)$ as in \eqref{eqphiT}. If $g$ is a Riemannian metric on $\mathbb{S}^n$ with respect to which all equators are minimal, then the pull-back metric $\phi^{*}g$ has also the same property. According to Theorem \ref{thmcorrespondenceAfull}, 
\begin{equation*}
	g=F_gk_g \quad \text{and} \quad  \phi^{*}g= F_{\phi^{*}g}k_{\phi^{*}g}, 
\end{equation*}
where
\begin{equation*}
	dV_{g} = (F_{g})^{\frac{n+1}{4}}dV_{\overline{g}}, \quad dV_{\phi^{*}g} = (F_{\phi^{*}g})^{\frac{n+1}{4}}dV_{\overline{g}},
\end{equation*} 
and both $k_g$, $k_{\phi^{*}g}$ belong to $\mathcal{K}^+_2(\mathbb{S}^n,\overline{g})$. All we need to show is that
	\begin{equation} \label{eqmainclaim}
		k_{\phi^{*}g} = k_g \cdot T,
	\end{equation}
where the right-hand side is defined by the rule \eqref{eqactionkilling}.

	Let $\delta(T)$ be the positive function on $\mathbb{S}^n$ uniquely defined by the requirement that $\phi(T)^{*}dV_{\overline{g}} = \delta(T)^{\frac{n+1}{4}}dV_{\overline{g}}$. Then, 
\begin{multline*}
	(F_{\phi^{*}g})^{\frac{n+1}{4}}dV_{\overline{g}}=dV_{\phi^{*}g} = \phi^{*}(dV_{g}) = \phi^*(F^{\frac{n+1}{4}}dV_{\overline{g}})\\ = (\phi^*F_g)^{\frac{n+1}{4}} \phi^{*}dV_{\overline{g}} = ((\phi^{*}F_{g})\delta(T))^{\frac{n+1}{4}} dV_{\overline{g}}.
\end{multline*} 
\noindent Hence, $F_{\phi*g}=(\phi^*F_g)\delta(T)$ and
\begin{equation} \label{eqpullback}
	k_{\phi^{*}g} = \frac{1}{F_{\phi^{*}g}}\phi^{*}g = \frac{1}{\delta(T)}\phi^{*}\left(\frac{1}{F_g}g\right) = \frac{1}{\delta(T)}\phi(T)^{*}k_{g}.
\end{equation}
\noindent \textbf{Claim 1}: $\delta(T)(p) = det(T)^{\frac{4}{n+1}}/|Tp|^4$ for every $p\in \mathbb{S}^n$. \\

\indent In fact, for every $p\in \mathbb{S}^n$ and $v\in T_p \mathbb{S}^n$,
\begin{equation*}
	d\phi(T)(p)v = \frac{1}{|Tp|}\left(Tv - \frac{\langle Tv,Tp \rangle}{|Tp|^2}Tp\right).
\end{equation*}
Let $dV_{euc}$ be the volume element of $\mathbb{R}^{n+1}$. Since $dV_{\overline{g}}(p)(v_1,\ldots,v_n)=dV_{euc}(p,v_1,\ldots,v_n)$ for all $p \in \mathbb{S}^n$ and $v_1,\ldots, v_n\in T_{p}\mathbb{S}^n$, we have
\begin{align*}
	(\phi(T)^{*}dV_{\overline{g}})(p)(v_1,\ldots,v_n) = & \\ 
& \hspace{-3cm}= dV_{euc}\left(\phi(T)(p),d\phi(T)(p)v_1,\ldots,d\phi(T)(p)v_n\right) \\ & \hspace{-3cm}= \frac{1}{|Tp|^{n+1}}dV_{euc}(Tp, Tv_{1}, \ldots, Tv_{n}) \\ 
& \hspace{-3cm}= \frac{det(T)}{|Tp|^{n+1}}dV_{euc}(p,v_1,\ldots,v_n) \\ & \hspace{-3cm}= \frac{det(T)}{|Tp|^{n+1}}dV_{\overline{g}}(p)(v_1,\ldots,v_n) .
\end{align*}
The claim follows. \\

\noindent \textbf{Claim 2}: For every $p \in \mathbb{S}^n$ and $k\in \mathcal{K}_{2}(\mathbb{S}^n,\overline{g})$,
\begin{equation*}
	(\phi(T)^{*}k)(p) = \frac{det(T)^{\frac{4}{n+1}}}{|Tp|^{4}} (k \cdot T)(p).
\end{equation*}

\indent It is enough to prove this claim for Killing symmetric two-tensors of the form $k=K\odot L$, $K, L\in \mathfrak{so}(n+1,\mathbb{R})$. Observing that $K(x)$, $L(x)\in T_x\mathbb{S}^n$ are orthogonal to $x \in S^{n}$, we compute, for every $p\in \mathbb{S}^n$ and $v,w\in T_{p}\mathbb{S}^n$,
\begin{align*}
	(\phi(T)^{*}(K\odot L))_p(v,w) = &\\ & \hspace{-3.3cm}=\left\langle K\left(\frac{Tp}{|Tp|}\right),d\phi(T)(p)v\right\rangle \left\langle L\left(\frac{Tp}{|Tp|}\right),d\phi(T)(p)w\right\rangle + \ldots \\ & \hspace{-3.3cm}= \left\langle K\left(\frac{Tp}{|Tp|}\right),\frac{Tv}{|Tp|}\right\rangle \left\langle L\left(\frac{Tp}{|Tp|}\right),\frac{Tw}{|Tp|}\right\rangle + \ldots \\
	& \hspace{-3.3cm}= \frac{1}{|Tp|^4} \left\langle T^{\top}K Tp,v\right\rangle \left\langle T^{\top}L Tp,w\right\rangle + \ldots \\ & \hspace{-3.3cm}= \frac{det(T)^{\frac{4}{n+1}}}{|Tp|^{4}} ((K\odot L)\cdot T)_p(v,w).
\end{align*}
(In the above computation, we hid the terms involving the permutation of $K$ and $L$ under the dots). The claim follows.  \\

Combing both Claims with equation \eqref{eqpullback}, we verify the validity of \eqref{eqmainclaim}, and the result follows.
\end{proof}

	The equivariance of the correspondence has the following immediate consequence:
\begin{cor} \label{corisometrygroup}
 	Let $g\in \mathcal{E}(\mathbb{S}^n)$ and $R\in Curv_+(\mathbb{R}^{n+1})$ be related by the maps defined in Theorem  \ref{thmcorrespondenceasalgorithm}. Let $T\in GL(n+1,\mathbb{R})$. Then, $R\cdot T=R$ if and only if the map $\phi(T)$ as in \eqref{eqphiT} is an isometry of $(\mathbb{S}^n,g)$.
\end{cor}

\begin{proof}
	In fact, $R\cdot T=R$ if and only if $\phi(T)^*g=g$.
\end{proof}
	
	As an application, we obtain:
	
\begin{cor} \label{corantipodalsymmetry}
	The antipodal map is an isometry of every Riemannian metric on the sphere $\mathbb{S}^n$ with respect to which all equators are minimal. 
\end{cor}

\begin{proof}
	Notice that the antipodal map is simply the map $A=\phi(-Id) : \mathbb{S}^n \rightarrow \mathbb{S}^n$. Since $det(-Id)=(-1)^{n+1}$ and every $R\in Curv(\mathbb{R}^{n+1})$ is four-linear, we have $R\cdot (-Id)=R$ for every $R\in Curv(\mathbb{R}^{n+1})$. By Corollary \ref{corisometrygroup}, every $g\in \mathcal{E}(\mathbb{S}^n)$ is such that $A^*g=g$.
\end{proof}

	In particular, for every $g\in \mathcal{E}(\mathbb{S}^n)$ and every $T\in GL(n+1,\mathbb{R})$, we have $\phi(-T)^*g=\phi(T)^*A^*g=\phi(T)^*g$. Hence, the group 
	\begin{equation*}
		\mathcal{G}_n=\tilde{\mathcal{G}}_n/\{\pm Id\}\simeq GL(n+1,\mathbb{R})/\{\lambda Id\,|\,\lambda \neq 0\} = PGL(n+1,\mathbb{R})
	\end{equation*}
	has a well-defined action on $\mathcal{E}(\mathbb{S}^n)$, given by $g\cdot [T]=\phi(T)^*g$ for every $[T]\in PGL(n+1,\mathbb{R})$ and every $g\in \mathcal{E}(\mathbb{S}^n)$.
	
	It is interesting to notice that different elements of $\mathcal{G}_n$ induce different permutations in the set of $(n-1)$-equators.
	
	It is also immediate to check that the $GL(n+1,\mathbb{R}^n)$-action on the spaces $\mathcal{K}_2(\mathbb{S}^n,\overline{g})$ and $Curv(\mathbb{R}^n)$ induce a $PGL(n+1,\mathbb{R}^n)$-action as well, because $\lambda T$ acts in the same way as $T$ for every $\lambda\neq 0$ and every $T\in GL(n+1,\mathbb{R})$. 
	
	Now that we have defined the action of $PGL(n+1,\mathbb{R})$ on each space $\mathcal{E}(\mathbb{S}^n)$ and $Curv_{+}(\mathbb{R}^{n+1})$, and checked, using Theorem \ref{thmequivariance}, its equivariance under the correspondence described in Theorem \ref{thmcorrespondenceasalgorithm}, the proof of the Classification Theorem, as stated in the Introduction, is finished.
	
\begin{rmk}\label{rmkrestrictequivariance}
	Equivariance assertions similar to those presented in this Section are true in the more general context of Theorem \ref{thmcorrespondenceAfull}, and can be proven by the same arguments. Namely, one just needs to consider the action of the subgroup of $GL(n+1,\mathbb{R}^n)$ consisting of those maps $T$ such that the diffeomorphism $\phi(T)$ as in \eqref{eqphiT} maps a given open subset $U\subset \mathbb{S}^n$ into itself. A similar observation can be made about Hangan's correspondence (Theorem 2 in \cite{Han1}). 
\end{rmk}

\begin{rmk}\label{rmkprojective}
	Let $\mathbb{RP}^n$ denote the $n$-dimensional real projective space obtained as the quotient of $\mathbb{S}^n$ by the antipodal map. The quotient of a $k$-equator of $\mathbb{S}^n$ is an embedded copy of $\mathbb{RP}^{k}$ in $\mathbb{RP}^n$ that we call a \textit{linear projective $k$-plane}. A consequence of Corollary \ref{corantipodalsymmetry} is that the pull-back by the standard projection $\pi : \mathbb{S}^n \rightarrow \mathbb{RP}^n$ defines a bijective correspondence between the set of Riemannian metrics on $\mathbb{RP}^n$ with respect to which all linear projective $(n-1)$-planes are minimal and the set $\mathcal{E}(\mathbb{S}^n)$. We explore the geometry of these metrics on $\mathbb{RP}^n$ a bit further in Remark \ref{rmkprojective2}.
\end{rmk}

\section{Three-dimensions} \label{sec3d}

\subsection{Uniqueness} \label{subsecuniq} The set of \textit{oriented} equators of $\mathbb{S}^3$ is a set of oriented smoothly embedded two-spheres of $\mathbb{S}^3$ that is para\-me\-tri\-sed by $\mathbb{S}^3$ itself, via the smooth map $v\mapsto \Sigma_v$. Given any point $p\in \mathbb{S}^3$ and any (oriented) two-dimensional linear subspace $\pi\subset T_p\mathbb{S}^3$, there exists a unique (oriented) equator $\Sigma_{(p,\pi)}\subset \mathbb{S}^3$ that passes through $p$ and is such that $\pi=T_p\Sigma_{(p,\pi)}$. Moreover, the map $(p,\pi) \mapsto \Sigma_{(p,\pi)}$ is smooth.

	Thus, on any $(\mathbb{S}^3,g)$ with minimal equators, the set of oriented equators constitute a \textit{transitive family} of oriented embedded minimal two-spheres in the sense of J. G\'alvez and P. Mira (see Definition 2.1 in \cite{GalMir}). We can therefore apply their Uniqueness Theorem (\cite{GalMir}, Theorem 2.3) to derive the following strong conclusion:

\begin{thm} \label{thmuniqueness}
	Let $g$ be a Riemannian metric on the sphere $\mathbb{S}^{3}$ with respect to which all two-equators are minimal. Any immersed minimal two-sphere in $(\mathbb{S}^3,g)$ must be an equator.
\end{thm}

	This statement cannot be generalised to dimensions $n>3$. For instance, $(\mathbb{S}^4,\overline{g})$ itself contains infinitely many embedded minimal three-sphe\-res that are not equators. These non-totally geodesic hyperspheres were first constructed by Wu-Yi Hsiang \cite{Hsi}. (See also \cite{HsiSte}, Theorem 4).

\subsection{Area, index and nullity of two-equators} \label{subsecnul}

\indent Let $\Sigma$ be a minimal orientable hypersurface of $(\mathbb{S}^n,g)$, oriented according to some choice of unit normal $N^g$. The Jacobi operator of $\Sigma$ is the symmetric elliptic operator $\mathcal{J}_g$, acting on functions $\eta\in C^{\infty}(\Sigma)$, that is given by
\begin{equation*}
	\mathcal{J}_g(\eta) = \Delta_{g}\eta + Ric_{g}(N^g,N^g)\eta+|A_g|_g^2\eta.
\end{equation*} 
Here $A_g$ is the second fundamental form of $\Sigma$ in $(\mathbb{S}^n,g)$, and $Ric_g$ is the Ricci tensor of $(\mathbb{S}^n,g)$.
	The Jacobi operator appears naturally in the study of the second variation of the area functional \cite{Sim}. In fact, for any smooth variation $\{\Sigma_t\}$ of $\Sigma_0=\Sigma$ with normal speed $\eta$ at $t=0$, 
	\begin{equation*}
		\frac{d^2}{dt^2}_{|_{t=0}}area(\Sigma_t,g) = -\int_{\Sigma} \eta \mathcal{J}_{g}(\eta) d\Sigma_g.
	\end{equation*}	
	
	The Morse index of $\Sigma$ is equal to the number of negative eigenvalues of $\mathcal{J}_g$ (counted with multiplicity). Thus, it measures the degree of instability of $\Sigma$ as a critical point of the area functional. 
	
	The nullity of $\Sigma$, on the other hand, is the dimension of the space of Jacobi functions, that is, the dimension of the kernel of $\mathcal{J}_g$. For a variation $\{\Sigma_t\}$ of a minimal hypersurface $\Sigma=\Sigma_0$ with normal speed $\eta$ at $t=0$, the mean curvature $H_t$ of $\Sigma_t$, viewed as a function on $\Sigma$, satisfies
	\begin{equation*}
		\frac{\partial}{\partial t}_{|_{t=0}} H_{t} = - \mathcal{J}_g(\eta).
	\end{equation*}
	Thus, a variation of $\Sigma$ by minimal hypersurfaces generates a Jacobi function on $\Sigma$. (It is not true in general, however, that a non-zero Jacobi function generates a deformation of this sort).

\begin{thm} \label{thmindexnullity}
	Let $g$ be a Riemannian metric on the sphere $\mathbb{S}^{3}$ with respect to which all equators are minimal two-spheres. Then every equator in $(\mathbb{S}^3,g)$ has Morse index one, nullity three, and the same area.
\end{thm}

\begin{proof} The theorem will be the consequence of successive claims:

$1)$ \textit{All equators have the same area}. In fact, the set of oriented minimal equators is smoothly parametrised by $\mathbb{S}^3$ itself. Since they are all critical points of the area functional, they all have the same area. 

$2)$ \textit{All equators have nullity at least three}. Given an equator $\Sigma_v$ with unit normal $N_g$, consider the linear map $K\in\mathfrak{so}(4,\mathbb{R}) \mapsto g(K,N_g)\in C^{\infty}(\Sigma_v)$. The kernel of this map, consisting of $K$ such that $K(p)\in T_{p}\Sigma_v$ for every $p\in \Sigma_v$, is a $3$-dimensional subspace of its 6-dimensional domain. Moreover, its $3$-dimensional image consists of Jacobi functions on $\Sigma_v$, because the flow of $K$ is a flow by elements of $SO(4)$ and therefore a flow that creates a variation of the equator $\Sigma_v$ by equators, which are all minimal in $(\mathbb{S}^3,\overline{g})$.

$3)$ \textit{All equators have index at least one.} The first eigenvalue of a symmetric elliptic operator as $\mathcal{J}_g$ has multiplicity one. By Claim 2, zero cannot be the first eigenvalue of $\mathcal{J}_g$.

$4)$ \textit{All equators have index one}. Recall that a minimal two-sphere in $(\mathbb{S}^3,g)$ must be an equator, by Theorem \ref{thmuniqueness}. By Claim 3, $(\mathbb{S}^3,g)$ satisfies the hypotheses of Marques-Neves Min-Max Theorem (see \cite{MarNev}, Theorem 3.4). As a consequence, all equators have, in particular, index equal to one.

$5)$ \textit{All two equators have nullity three.} In fact, by a general result of Shiu Yuen Cheng, on a two-sphere, three is an upper bound for the multiplicity of the second eigenvalue of a symmetric elliptic operator of the form of $\mathcal{J}_g$ (\textit{Cf}. \cite{Che}, Section 3). By Claim 4, all eigenvalues of $\mathcal{J}_g$ apart from the first are non-negative. Thus, the nullity of every equator is at most three. By Claim 2, we conclude that the nullity of every equator is precisely three, and the proof of the theorem is finished.
\end{proof}

\begin{rmk}
	The proof of Claim 4 above, based on Marques-Neves Theorem, actually shows that the common area of the minimal equators in $(\mathbb{S}^3,g)$ is equal to the \textit{Simon-Smith width} of this space. This number is the simplest geometric invariant associated to the min-max theory of the area functional on the space of two-dimensional spheres embedded in $\mathbb{S}^3$. (See \cite{ColdeL} for an introduction to this theory).

	The topology of the space of embedded spheres in $\mathbb{S}^3$ gives origin to other three min-max numbers. In \cite{AmbMarNev2}, we showed that, actually, all four spherical area widths of $(\mathbb{S}^3,g)$ with minimal equators are equal to the common area of the minimal equators. The converse, however, is not true. (\textit{Cf}. Theorem B in \cite{AmbMarNev2} for further elaboration).
\end{rmk}

\subsection{Isometries} \label{subsecisom} If $\phi$ is an isometry of a metric $g\in \mathcal{E}(\mathbb{S}^n)$, it sends each equator into a minimal hypersphere of $(\mathbb{S}^n,g)$, which a priori is not guaranteed to be another equator. In this regard, recall the aforementioned non-equatorial minimal three-spheres in $(\mathbb{S}^{4},\overline{g})$ constructed by Hsiang \cite{Hsi}. 

	Nevertheless, in dimension $n=3$, we can show that isometries of metrics in $\mathcal{E}(\mathbb{S}^n)$ map equators into equators. More precisely:
\begin{thm} \label{thm3disometries}
	Let $g_1$ and $g_2$ be Riemannian metrics on the sphere $\mathbb{S}^3$ with respect to which all equators are minimal. If 
	\begin{equation*}
		\phi : (\mathbb{S}^3,g_1) \rightarrow (\mathbb{S}^3,g_2)
	\end{equation*}
	is an isometry, then $\phi=\phi(T)$ as in \eqref{eqphiT} for some $T\in GL(n+1,\mathbb{R})$.
\end{thm}
\begin{proof}
	By Theorem \ref{thmuniqueness}, the image of any equator by the isometry $\phi$ must be an equator as well, that is, $\phi$ permutes equators. The theorem is now an immediate consequence of Proposition \ref{propnaturalsymmetries}.
\end{proof}
	
	Thus, two metrics in $\mathcal{E}(\mathbb{S}^3)$ are isometric to each other if, and only if, they lie in the same orbit by the action of $GL(n+1,\mathbb{R})$.
	
	Moreover, combining Theorem \ref{thm3disometries} with Corollary \ref{corisometrygroup}, we obtain a simple description of the isometry group of metrics $g\in \mathcal{E}(\mathbb{S}^3)$. In order to make a neat statement, recall that the stabilizer of a point in a space where a Lie group $H$ acts is the closed subgroup of $H$ consisting of those elements that fix the given point.

\begin{thm} \label{thm3disometrygroup}
Let $(\mathbb{S}^3,g)$ be a sphere with minimal equators, and let $R_g\in Curv_+(\mathbb{R}^n)$ be the algebraic curvature tensor corresponding to $g$ via Theorem \ref{thmcorrespondenceAfull}. The map $T\in GL(n+1,\mathbb{R})\mapsto \phi(T)$ defined by (\ref{eqhomomorphism2}) induces isomorphisms
\begin{equation*}
	Stab_{GL(4,\mathbb{R})}(R_g) \simeq Isom(\mathbb{S}^3,g) \quad \text{and} \quad Stab_{SL(4,\mathbb{R})}(R_g) \simeq Isom_0(\mathbb{S}^3,g),
\end{equation*}
which characterise the isometry group of $(\mathbb{S}^3,g)$ and the connected component of the identity in this group, respectively.
\end{thm}
\begin{proof}
	By Theorem \ref{thm3disometries}, any isometry of $(\mathbb{S}^3,g)$ must be of the form $\phi(T)$ for some $T\in GL(4,\mathbb{R})$. The statement about $Isom(\mathbb{S}^n,g)$ is then an immediate consequence of Corollary \ref{corisometrygroup}.
	
	Any element of $Isom_0(\mathbb{S}^3,g)$ is the form $\phi(T)$ for some $T\in GL(4,\mathbb{R})$ with positive determinant, because $\phi(T)$ is orientation-preserving if and only if $det(T)>0$. In particular, after multiplying $T$ by a positive number, $\phi(T)=\phi(T')$ for a unique $T'\in GL(4,\mathbb{R})$ with $det(T')=1$. But this $T'$ is therefore an element of the special linear group $SL(4,\mathbb{R})\subset GL(n+1,\mathbb{R})$ that fixes $R_g$. The second part of the statement follows.
\end{proof}

	In view of Theorem \ref{thm3disometrygroup}, the computation of $Isom(\mathbb{S}^3,g)$ is an algebraic problem. But it is a demanding one from the computational point of view. In \cite{AmbMarNev}, we focused on computing the dimension of this group. By studying the linearized action, we were able to find metrics in $\mathcal{E}(\mathbb{S}^3)$ with discrete isometry groups. (See Section 9.4 of \cite{AmbMarNev}). In particular, these metrics provided examples of metrics on $\mathbb{S}^3$, that are arbitrarily close to the canonical metric, and that  contain smooth $3$-parameter families of minimal spheres, which are not produced by smooth families of ambient isometries moving around a given minimal sphere. (These metrics are relevant to a question raised by Shing-Tung Yau in \cite{Yau}, p. 248. For more on it, see \cite{AmbMarNev}, Theorem C).
		
\begin{rmk} \label{rmkprojective2}
 Let $g$ be a Riemannian metric on $\mathbb{RP}^3$ such that all linear projective two-planes are minimal, as in Remark \ref{rmkprojective}. Pulling back such metric to $\mathbb{S}^3$ as in that remark, we can use Theorem \ref{thmuniqueness} to deduce the following corollary:  any embedded minimal projective plane in $(\mathbb{RP}^3,g)$ must be a linear projective two-plane. \\
\indent Moreover, every linear projective plane has the same area, equal to the least area of embedded projective planes in $(\mathbb{RP}^3,g)$. In fact, the linear projective planes form a path-connected set of minimal surfaces varying smoothly in $(\mathbb{RP}^3,g)$, so that they have the same area. By Proposition 2.3 in \cite{BraBreEicNev}, a least area embedded projective two-plane $\Sigma$ exists in $(\mathbb{RP}^3,g)$. By the above uniqueness assertion, $\Sigma$ must be one of the linear projective two-planes. (This generalises \cite{Gil}, where a similar reasoning was made for a Berger metric on $\mathbb{RP}^3$). \\
\indent These observations should be contrasted with the following result. The \textit{Finsler} metrics on $\mathbb{RP}^n$ with respect to which all linear projective hyperplanes minimise the Holmes-Thompson area in their \textit{homology class} have been classified by Juan Carlos \'Alvarez-Paiva and Emmanuel Fernandez in terms of smooth positive measures on the Grassmanian of two-planes of $\mathbb{R}^{n+1}$ (see \cite{AlvFer}, Theorem 8.5). In their classification, however, it is not immediately clear how to distinguish the \textit{Riemannian} metrics among them, or how to understand geometric properties of the Finsler metric out of the properties of these measures.
\end{rmk}

\section{Perspectives}\label{sex.perspectives}

\subsection{Further geometric explorations} While Theorem \ref{thmindexnullity} determined the variational properties of minimal spheres in $(\mathbb{S}^3,g)$ with minimal equators, there are several other interesting geometric problems that can be posed about them. 

	For instance, do they minimise area among surfaces that bound the same fraction of the total volume, as equators do in $(\mathbb{S}^3,\overline{g})$? 

	Using Theorems \ref{thm3disometries} and \ref{thm3disometrygroup}, it is also possible to investigate the moduli space $\mathcal{E}(\mathbb{S}^3)/ \sim$, where metrics are identified if they differ by isometry and scaling. Does it have an interesting topological or geometric structure? In which ways a sequence of metrics in $\mathcal{E}(\mathbb{S}^n)$ degenerate? Does a generic metric in $\mathcal{E}(\mathbb{S}^3)$ have an isometry group that contains only the antipodal map? And what are the possible isometry groups of metrics in $\mathcal{E}(\mathbb{S}^3)$?

	In higher dimensions, the analogous questions have the same relevance. Notice, however, that the appropriate versions of Theorems \ref{thmindexnullity}, \ref{thm3disometries} and \ref{thm3disometrygroup} for metrics in $\mathcal{E}(\mathbb{S}^n)$, in dimensions $n\geq 4$, are still lacking, as their proof made use of propositions that are specific of three dimensions, for instance Theorem \ref{thmuniqueness}.

\subsection{Zoll-like metrics in minimal hypersurface theory} One of our main motivations to understand spheres with minimal equators in dimensions $n\geq 3$ was that they would be the simplest, non-trivial examples of Riemannian metrics on the sphere that admit a \textit{wealth of minimal hyperspheres}, that is, a family of embedded minimal codi\-men\-sion-one spheres  $\{\Sigma_\sigma\}$ in $(\mathbb{S}^n,g)$, smoothly parametrized by points $\sigma\in \mathbb{RP}^n$, that satisfy the following axiom: for every point $p\in \mathbb{S}^n$ and for every tangent hyperplane $\pi\subset T_p\mathbb{S}^n$, there exists a unique $\sigma\in \mathbb{RP}^n$ such that $p\in \Sigma_\sigma$ and $\pi=T_p \Sigma_\sigma$. This concept, introduced by F. Marques, A. Neves and the author in \cite{AmbMarNev} (where we called it a \textit{Zoll family of minimal hyperspheres}), captures the essential features of the set of geodesics of a \textit{Zoll metric} on $\mathbb{S}^2$, that is, of a metric whose geodesics are all closed, simple and have the same length. (The comprehensive and still quite up-to-date reference on the subject of Zoll metrics is the excellent book by Arthur Besse \cite{Bes}).

	At the beginning of our investigations, the only non-trivial examples of such metrics that we knew were the homogeneous metrics on $\mathbb{S}^3$. In fact, while Francisco Torralbo \cite{Tor} explicitly observed that the Berger metrics, in their standard presentations as left-invariant metrics, are metrics with minimal equators, the general classification theorem of immersed constant mean curvature spheres in all homogeneous three-dimensional spheres by Meeks-Mira-Perez-Ros \cite{MeeMirPerRos} can be used to show that their spaces of minimal two-spheres form a wealth of minimal two-spheres in the above sense. However, it was not so clear from their work how exactly these minimal spheres looked like inside $\mathbb{S}^3$. Since, up to isometries, homogeneous metrics on $\mathbb{S}^3$ are  precisely the left-invariant metrics, it was a pleasant surprise to realise that \textit{all} left-invariant metrics on $\mathbb{S}^3$ are metrics with minimal equators! See Example \ref{exampleleftinvar}.

	In any case, in all dimensions $n\geq 3$, the metrics with minimal equators form a thin part of the space of metrics on $\mathbb{S}^n$ that contain a wealth of minimal hyperspheres. This fact follows from our constructions of such metrics by a perturbation method, developed in \cite{AmbMarNev}. More recently, Diego Guajardo and the author refined the construction, and obtained further examples with controlled isometry groups and that are isometrically embeddable as hypersurfaces of $\mathbb{R}^n$, much like Otto Zoll's original constructions in $\mathbb{S}^2$. (But we do not know if these metrics of revolution we construct in dimensions $n\geq 3$ are analytic).\footnote{I thank Alberto Abbondandolo for bringing to my attention that the analyticity of the original Zoll spheres of revolution was the most important element of the surprise his discovery provoked among experts at the time. O. Zoll was a student of D. Hilbert, and the existence of such geometries was the main result of his PhD thesis, published in \cite{Zol}.}
	
	An important question about a generalised Zoll metric on $\mathbb{S}^n$ is whether its wealth of minimal hyperspheres is unique. As we have seen in Section \ref{subsecuniq}, a stronger uniqueness statement is valid in dimension $n=3$ which does not hold in higher dimensions. Investigating this uniqueness problem in the class metrics with minimal equators (or even in $\overline{g}$) seems to be a promising start. Notice that a positive answer would allow to extend many results of Section \ref{subsecisom} to all dimensions $n\geq 4$.
	
	Another relevant question is whether a wealth of minimal hyperspheres is non-degenerated, in the sense that all Jacobi functions of each member of the wealth of minimal hyperspheres are the normal velocities of variations by members of the wealth. (In other words, they form a \textit{Morse-Bott family of critical points} of the area functional). The family of equators in $(\mathbb{S}^n,\overline{g})$ is non-degenerated in this sense, in all dimensions. Again, it may be worth to investigate this question for the class of metrics with minimal equators, which could be a step towards the generalisation of Theorem \ref{thmindexnullity} to all dimensions.
	
	In order to further explore the space of generalised Zoll metrics on $\mathbb{S}^n$, we think it would be important to establish general properties of the \textit{Funk-Radon transform} associated to the family of equators in a sphere with minimal equators $(\mathbb{S}^n,g)$. This is the map that assigns to each $f\in C^{\infty}(\mathbb{S}^n)$ the function $\mathcal{R}(f)\in C^{\infty}(\mathbb{RP}^n)$ such that
	\begin{equation*}
		\mathcal{R}(f)(\sigma) = \int_{\Sigma_v} f(x)\, dA_{g}(x),
	\end{equation*}
where $\sigma=[\pm v]$.

	The Radon transform is very well-understood in case of the canonical metric, because it is equivariant under the action of $SO(n+1)$, and the representation theory of this group can thus be applied. We think that the analysis of the Radon transform of non-homogeneous metrics with minimal equators will present interesting challenges.
	
	If the family of equators is Morse-Bott and this transform is surjective and has a well characterised kernel for some $(\mathbb{S}^n,\overline{g})$ with minimal equators, there is hope that the techniques used in \cite{AmbMarNev} and \cite{AmbGua} in order to perturb the canonical metric $\overline{g}$ into new metrics on $\mathbb{S}^n$ admitting a wealth of minimal hyperspheres can be adapted, in a straightforward fashion, to yield new such metrics by perturbation of $g$.
	
	Finally, let $A(g)$ denote the common area of the equators of $(\mathbb{S}^n,g)$ with minimal equators. It would be interesting to know if there is some probability measure $\mu$ on $\mathbb{RP}^n$ such that
	\begin{equation*}
		\int_{\mathbb{RP}^n} \mathcal{R}(f)(\sigma) d\mu(\sigma) = \frac{A(g)}{vol(\mathbb{S}^n,g)}\int_{\mathbb{S}^n} f(p)dV_g(p).
	\end{equation*}
	The Radon transform of homogenous metrics on $\mathbb{S}^3$ have this property (see \cite{AmbMon2}, Theorem 3.1). If this is the case in general, we suspect that there is a good chance that metrics with minimal equators are local maxima, among metrics with the same volume in their conformal class, of the functional that computes, for each metric $g$ in $\mathbb{S}^n$, the infimum of the area of minimal hyperspheres in $(\mathbb{S}^n,g)$. (\textit{Cf}. \cite{AmbMon}, Proposition 1.4.1)
	
\subsection{Manifolds with positive sectional curvature} Let $(M^{n+1},h)$ be a Riemannian manifold of dimensions $n+1\geq 4$. On each tangent space $T_pM$, we can use the Classification Theorem to associate the curvature tensor $R_p$ to a Riemannian metric $g_p$ on the unit sphere $\mathbb{S}^n\subset (T_pM,h_p)$ with respect to which all equators are minimal.

	A very curious problem is, therefore, to explore the properties of the map that assigns to each $p\in M$ the isometry class of the metric $g_p$. One way to study it is to choose some geometric invariant from the metrics $g_p$ - say, the common area of their equators or their first Laplace eigenvalue - and study the global behaviour of the function on $M$ thus obtained. A tantalising possibility is that these functions may reveal new facets of manifolds with positive sectional curvature. Or that they may evolve in nice ways under the Ricci flow.
	
	A related problem (suggested to us by Gavin Ball) would be to classify those (complete) Riemannian manifolds with positive sectional curvature such that the curvature tensors on every point define the same isometry class of metrics with minimal equators. Are they just the homogeneous manifolds with positive sectional curvature?

\appendix
\setcounter{secnumdepth}{0}

\section{Appendix A - An inverse problem}
	
	The problem discussed in this survey is an example of an inverse problem in a geometric, variational context: to understand all the geometries on a given space such that a geometric functional associated to them have a prescribed set of critical points. There is a long list of problems that fit in this description; we point out one that has been very influential. As the Fourth Problem of his famous list \cite{Hil}, David Hilbert proposed to \textit{construct and study} the geometries in which the straight line segment is the shortest connection between two points. His vague formulation is of course open to multiple interpretations (\textit{e.g.} what is a geometry?), and has motivated fruitful mathematical investigations, many of which are described in the excellent surveys by J. \'Alvarez-Paiva \cite{Alv} and Athanase Papadopoulos \cite{Pap}.

	One of the inspirations for Hilbert was a theorem of Beltrami. In 1865, Beltrami published the results of his investigations on the problem of determining all (connected) surfaces that admit an atlas of local charts on which their (unparametrized) geodesics are represented by straight lines \cite{Bel1}. Despite his stated initial hope to find non-trivial examples of such surfaces, Beltrami eventually proved that, if such atlas exists, then the surface has constant (positive, zero or negative) Gaussian curvature.
	
	Even more, Beltrami reached this conclusion after ingenious computations led him to a simple formula for such metrics on those local charts, possibly after an affine change of coordinates. Beltrami's formula was
	\begin{equation*} 
		R^2\frac{(v^2+a^2)du^2 - 2uvdudv + (u^2+a^2)dv^2}{(u^2+v^2+a^2)^2},
	\end{equation*}
where $R$ and $a$ are constants. This is a metric with Gaussian curvature equal to $1/R^2$. While Beltrami, in his computations, explicitly considered as valid constants also imaginary numbers, the motivation from Car\-to\-gra\-phy led him to dismiss this possibility. 

	Interestingly enough, three years later, when Beltrami published his famous essay about hyperbolic geometry \cite{Bel2}, the complete hyperbolic metric he defined on the unit disc $u^2+v^2<1$ was precisely the above, with the choice of constants $R=a=i$. This was one of the origins of what is known today as the as the Beltrami-Cayley-Klein (projective) model of the abstract hyperbolic plane.

	Ludwig Schl\"afli \cite{Sch} generalised Beltrami's Theorem for all dimensions $n\geq 3$ (see also Beltrami's observations about Schl\"afli's paper \cite{Bel3}). More precisely, he proved that a (connected) Riemannian $n$-dimensional manifold covered by charts that represent its (un\-pa\-ra\-me\-tri\-zed) geodesics by straight lines has constant sectional curvature. A related theorem of Élie Cartan (see \cite{Car}, Chapter VI, Section IX) asserts that a Riemannian manifold has constant sectional curvature if and only if, for a fixed $k\geq 2$, through every point and through every $k$-dimensional subspace of the tangent space at that point passes a $k$-dimensional totally geodesic submanifold tangent to it. (For a modern proof, see Theorem 1.16 in \cite{DajToj}). 

	Perhaps surprisingly, the problem of characterising Riemannian metrics on $\mathbb{R}^{n}$ with respect to which all $k$-planes are minimal, for a given $k=2,\ldots n-1$, did not attract the attention of geometers until much more recently.\footnote{An exception we could find is the work of Hermann Busemann \cite{Bus}. However, his formulation was more in the spirit of Hilbert's problem, and did not concern only geometries determined by Riemannian metrics. In fact, in the introduction of \cite{Bus}, Busemann expresses the view that the question of investigating all such geometries, in that general set-up, is not fruitful, by virtue of the fact that there are simply too many of them.} 

	This problem was essentially solved in two papers by Hangan. In \cite{Han2}, he characterised all Riemannian metrics in open subsets of $\mathbb{R}^n$, $n\geq 3$, with respect to which all $k$-dimensional planes (for some fixed $2\leq k\leq n-2$) that intersect the given set are minimal. By skilful tensorial computations, Hangan proved that such metrics must also have constant sectional curvature. 
	
	Before we discuss the case $k=n-1$, dealt with by Hangan in \cite{Han1}, it is a good moment to summarize all the results above as follows:
	
\begin{thm*}[Beltrami \cite{Bel1}, Schl\"afli \cite{Sch}, Hangan \cite{Han2}] \label{thmbelschhan} Let $n\geq 2$ be an integer, and either $k=1$ or $k\leq n-2$ be a positive integer. If $g$ is a Riemannian metric on a connected non-empty open subset $U$ of $\mathbb{R}^{n}$ with respect to which all $k$-dimensional planes intersecting $U$ are minimal in $(U,g)$, then $g$ has constant sectional curvature.
\end{thm*}
	The Beltrami-Cayley-Klein metric on the ball in $\mathbb{R}^n$, the Euclidean metric itself, and the push-forward of the canonical metric on the sphere $\mathbb{S}^n$ to $\mathbb{R}^n$ by a central projection provide examples with all possible values of constant sectional curvature (up to scaling).
	
	Furthermore, since central projections map (pieces of) $k$-equators in $\mathbb{S}^n$ into $k$-planes in $\mathbb{R}^n$ for every $k=1,\ldots,n-1$, it follows from the Theorem above that, for each fixed integer $n\geq 2$ and each positive integer $k\leq n-2$ or $k=1$, a Riemannian metric on $\mathbb{S}^n$ with respect to which all $k$-equators are minimal must have constant positive sectional curvature. Thus, the focus on the case $k=n-1\geq 2$ was justified. \\
	
	What else can be said about the case $k=n-1\geq 2$? Despite the classification result of Hangan (\cite{Han1}, Theorem 2), it seems that the systematic study of the geometric properties of the metrics on open domains of $\mathbb{R}^{n}$ with respect to which all hyperplanes are minimal has not been carried out yet.
	
	The subject came into focus in the nineties after the works of Mohammed Bekkar, who observed that the Heisenberg metric
	\begin{equation*}
	g = dx^2 + dy^2 + (dz + ydx - xdy)^2
	\end{equation*}
on $\mathbb{R}^3$ is a homogeneous metric with respect to which all planes are minimal \cite{Bek1}. Robert Lutz then raised the question of what metrics on (open domains of) $\mathbb{R}^n$ have minimal hyperplanes (see \cite{Bek3}).\footnote{In the paper \cite{Han1}, where Hangan gives his solution to the problem, he also acknowledges Lutz as proposing the question.} 

	Bekkar started the systematic study of such metrics when $n=3$ (\cite{Bek2}, \cite{Bek3}). In particular, he classified those that are axy-symmetric, like the Heinseberg metric, and computed the formal dimension of the set of solutions near the Euclidean metric.\footnote{In \cite{Bek3}, Bekkar comments that Robert Bryant studied the problem around the same the time, and in particular he also computed the dimension of the space of metrics with minimal planes in $\mathbb{R}^3$ to be 20.} Hangan gave a different characterisation of the Heisenberg metric under an algebraic assumption (\cite{Han1}, Theorem 3), and observed in \cite{Han2} that the generalised Heisenberg spaces introduced by Aroldo Kaplan \cite{Kap} are homogeneous examples of such metrics on $\mathbb{R}^{n}$, for all dimensions $n\geq 4$.

	It would be interesting, for instance, to investigate which of such metrics on open connected subsets $U\subset \mathbb{R}^n$ are complete. This geometric restriction probably imposes restriction on the open subset itself (for instance, see Remark \ref{rmkmaximaldomain}). Also, in particular, it would be interesting to find non-homogeneous examples, and study their behaviour near infinity. 
	
	Regarding their minimal hyperplanes, we think that an interesting question would be to decide, for instance, under which further conditions the hyperplanes have finite total (intrinsic) curvature, and under which meaningful criteria they can be classified among area-minimising surfaces or among minimal planes. (They are certainly area-minimising surfaces, because parallel hyperplanes are minimal and foliate the space. Are they calibrated in the sense of \cite{HarLaw}? We remark that the three-dimensional Heisenberg space contains non-planar entire minimal graphs \cite{FerMir}, which have non-positive Gaussian curvature and at most quadratic intrinsic area-growth \cite{Man}). 
		
\section{Appendix B - Other model geometries}

	An equator of $\mathbb{S}^n$ is uniquely determined once a point in the Grassmanian of unoriented tangent hyperplanes of $\mathbb{S}^n$ is given, and it is moreover totally geodesic for the canonical metric. As we saw, the problem of understanding metrics with respect to which this particular family of hypersurfaces are all minimal led to interesting results. Thus, it is natural to pose similar questions about other model spaces that contain natural families of submanifolds that are, on one hand, uniquely determined by points in some special submanifold of their Grassmanian of unoriented tangent $k$-planes, and which are, one the other hand, totally geodesic with respect to some canonical geometry.
	
	The projective spaces $\mathbb{RP}^n$, $\mathbb{CP}^n$, $\mathbb{HP}^n$	and the Cayley plane $Ca\mathbb{P}^2$, endowed with their standard symmetric metrics (that is, the compact rank one symmetric spaces other than $\mathbb{S}^n$), together with the families of their projective linear subspaces, are, therefore, the next model geometries to look at. 
	
	Since metrics with minimal equators have antipodal invariance (by Corollary \ref{corantipodalsymmetry}), the study of this problem in $\mathbb{RP}^n$ is subsumed in the study of the problem in $\mathbb{S}^n$.
	
	In \cite{Luz}, Luciano L. Junior made a systematic study of almost Hermitian structures on $\mathbb{CP}^m$ such that all $\mathbb{CP}^k$, for a given $k=1,\ldots,m-1$, are \textit{minimal and complex} submanifolds. For instance, he proved that, when $m\geq 3$ and $k=m-1$, the necessary and sufficient condition is that the complex structure is integrable and the K\"ahler form $\omega(\cdot,\cdot)=g(J\cdot,\cdot)$ satisfies $d\omega^{n-1}=0$. These are the so called \textit{balanced} Hermitian structures \cite{Mic}, and most of them are not Kähler.

	To our best knowledge, the case of the other projective spaces has not been investigated yet, nor the case of the symmetric metrics on their non-compact companions. (For the case of the Euclidean space, see Appendix A).
	
	Aside from projective spaces, other relevant model geometries are those admitting calibrations \cite{HarLaw}. See for instance the works \cite{BalMad1} and \cite{BalMad2} of Gavin Ball and Jesse Madnick.

\end{document}